\documentclass{amsart}
\usepackage{amsaddr}

\usepackage{hyperref}
\hypersetup{colorlinks,%
citecolor=black,%
filecolor=black,%
linkcolor=blue,%
urlcolor=black}
\usepackage[capitalize,nameinlink,noabbrev]{cleveref}

\usepackage{thmtools}
\usepackage{thm-restate}
\usepackage{nima}

\usepackage{tikz}  
\usetikzlibrary{arrows}

\newenvironment{subproof}[1][\proofname]{%
  \begin{proof}[#1]%
}{%
  \end{proof}%
}

\DeclareMathOperator{\CAT}{CAT}
\DeclareMathOperator{\Rips}{Rips}

\title[Hyperbolic Relative to Strongly Shortcut]{Relatively hyperbolic
  groups with strongly shortcut parabolics are strongly shortcut}

\author{Nima Hoda}
\address{École normale supérieure, Université
  PSL, CNRS, Paris, France \\ \ \\
  Instytut Matematyczny,
  Uniwersytet Wroc\l awski\\
  pl.\ Grun\-wal\-dzki 2/4,
  50--384 Wroc{\l}aw, Poland \\ \ \\
  Cornell University}
\email{nima.hoda@mail.mcgill.ca}

\author{Suraj Krishna M S}
\address{Faculty of Mathematics, Technion -- Israel Institute of Technology \\
  Haifa 32000, Israel}
\email{surajms@campus.technion.ac.il}
\date{\today}

\keywords{relatively hyperbolic group, %
  strongly shortcut group, %
  strongly shortcut space}
\subjclass[2010]{20F65, 
  20F67} 

\begin{document}

\begin{abstract}
We show that a group that is hyperbolic relative to strongly shortcut groups is itself strongly shortcut, thus obtaining new examples of strongly shortcut groups. The proof relies on a result of independent interest: we show that every relatively hyperbolic group acts properly and cocompactly on a graph in which the parabolic subgroups act properly and cocompactly on convex subgraphs.
\end{abstract}

\maketitle

\tableofcontents

\section{Introduction}

Strongly shortcut graphs and groups were introduced by the first named
author \cite{Hoda:2018b} who later generalized the
strong shortcut property to rough geodesic metric spaces
\cite{Hoda:shortcut_space}.  The strong shortcut property is a very
general form of nonpositive curvature
{condition} satisfied by many spaces
of interest in geometric group theory, metric
graph theory and geometric topology.  These include Gromov-hyperbolic spaces \cite{Hoda:2018b}, asymptotically $\CAT(0)$ spaces \cite{Hoda:shortcut_space}, hierarchically hyperbolic spaces, coarse Helly metric spaces of uniformly bounded geometry \cite{hhs_strongly_shortcut}, $1$-skeletons of finite dimensional $\CAT(0)$ cube complexes (i.e. median graphs), $1$-skeletons of quadric complexes (i.e. hereditary modular graphs), $1$-skeletons of systolic complexes (i.e. bridged graphs), standard Cayley graphs of Coxeter groups \cite{Hoda:2018b} and all of the Thurston geometries except Sol \cite{Hoda:heisenberg, Kar:2011}.  Despite this surprisingly unifying nature, there are nonetheless important consequences for groups that act metrically properly and coboundedly on strongly shortcut geodesic metric spaces: finite presentability, polynomial isoperimetric function and thus decidable word problem \cite{Hoda:2018b, Hoda:shortcut_space}.

The strong shortcut property is essentially about limitations on the scale and precision at which subspaces can approximate circles.  Specifically:
\begin{defn}[Strongly shortcut]
A graph $\Gamma$ is \emph{strongly shortcut} if, for some $K > 1$ there is a bound on the lengths of the $K$-bilipschitz combinatorial cycles in $\Gamma$.  A group $G$ is \emph{strongly shortcut} if $G$ acts properly and cocompactly on a strongly shortcut graph.\end{defn}

This turns out to be equivalent to the existence of a metrically proper and cobounded $G$-action on a strongly shortcut 
geodesic metric space, which we define in \cref{sec:cayley_with_ss_parabolics}.  Thus the following classes of groups are all strongly shortcut: hyperbolic groups \cite{Gromov:1987}, asymptotically $\CAT(0)$ groups \cite{Kar:2011} (e.g. $\CAT(0)$ groups \cite{Bridson:1999}), hierarchically hyperbolic groups \cite{hhs1,hhs2} (e.g. mapping class groups of surfaces \cite{gcs1,gcs2}), coarse Helly groups \cite{chalopin2020helly} (e.g. Artin groups of FC-type, weak Garside groups \cite{huang2019helly}), the discrete Heisenberg group \cite{Hoda:heisenberg},  systolic groups (e.g. finitely presented $C(6)$ small cancellation groups \cite{Wise:2003}) and quadric groups (e.g. $C(4)$-$T(4)$ small cancellation groups) \cite{Hoda:2017}.

Our main result is the following.
{
\begin{restatable}{thrm}{main}\label{ss_closed_under_relhyp}
 Let $G$ be a
  finitely generated group that is hyperbolic
  relative to strongly shortcut groups.  Then $G$ is strongly
  shortcut.   
\end{restatable}
}
\cref{ss_closed_under_relhyp} allows us to obtain examples of strongly shortcut groups that are not known to be strongly shortcut by any other means.  For example, let $G$ be the free product of two copies of the discrete Heisenberg group and let $\langle t \rangle$ be a maximal cyclic subgroup generated by a loxodromic element $t$ of the Bass-Serre tree of $G$.  Then the amalgamated free product $G \ast_{\langle t \rangle} G$ is hyperbolic relative to discrete Heisenberg subgroups by Dahmani \cite{dahmani_combination} and thus is strongly shortcut by \cref{ss_closed_under_relhyp} and \cite{Hoda:heisenberg}.

Our approach to proving \cref{ss_closed_under_relhyp} is to use properties of asymptotic cones of strongly shortcut groups and relatively hyperbolic groups. A result of the first named author characterizes strongly shortcut groups as those whose asymptotic cones have no isometrically embedded circles (\cite[Theorem~3.7]{Hoda:shortcut_space}), while a result of Osin and Sapir \cite[Theorem~A.1]{Drutu:2005} guarantees that asymptotic cones of relatively hyperbolic groups are tree-graded. Thus, any isometrically embedded circle in an asymptotic cone of a relatively hyperbolic group has to be contained in a piece, which is impossible if the peripherals are strongly shortcut.

In the course of the proof of \cref{ss_closed_under_relhyp} we
restrict the combinatorial horoball construction of Groves and Manning
\cite{groves_manning_horoball} to a sufficiently large finite number
of levels, thus obtaining the following result which may be of
independent interest.

\begin{restatable}{thrm}{convexify}
    \label{convexify_horoball} 
  Let $G$ be a finitely generated group that is hyperbolic relative to
  finitely generated subgroups $(H_i)_i$.  For each $i$, let $S_i$ be
  a finite generating set for $H_i$.  Then there is a connected, free
  cocompact $G$-graph $\Gamma$ with subgraphs $(\Gamma_i)_i$ such
  that, for each $i$,
  \begin{enumerate}
      \item $\Gamma_i$ is a Rips graph of $\Cay(H_i,S_i)$,
      \item $H_i$ stabilizes $\Gamma_i$,
      \item the $H_i$ action on $\Gamma_i$ is free and cocompact, and
      \item $\Gamma_i$ is convex in $\Gamma$.
  \end{enumerate}
\end{restatable}

We use \cref{convexify_horoball} to prove \cref{ss_parabolics_in_cayley}, which says that $G$ has a Cayley graph in which the $H_i$ are strongly shortcut subspaces.  

\subsection*{Structure of the paper}
In \cref{basic_defs}, we recall the Groves and Manning combinatorial horoball construction and their characterization of relative hyperbolicity. \cref{sec:convexity} is devoted to the proof of \cref{convexify_horoball}. In \cref{sec:cayley_with_ss_parabolics}, we show that a relatively hyperbolic group with strongly shortcut parabolics admits a Cayley graph in which the parabolics are strongly shortcut subspaces. Finally, we recall the notion of asymptotic cones and prove the main result \cref{ss_closed_under_relhyp} in \cref{sec:main_result}.

\subsection*{Acknowledgements}
This work was supported by Polish Narodowe Centrum Nauki
UMO-2017/25/B/ST1/01335 as well as by the grant 346300 for IMPAN from
the Simons Foundation and the matching 2015-2019 Polish MNiSW fund.
The collaboration that led to this article was initiated at the 2019
Simons Semester in Geometric and Analytic Group Theory in Warsaw.

The first named author was supported by the ERC grant GroIsRan and an
NSERC Postdoctoral Fellowship.  The second named author was supported
by CEFIPRA grant number 5801-1, ``Interactions between dynamical
systems, geometry and number theory'' at Tata Institute of Fundamental
Research and by grant number ISF 1226/19 at the Technion.

We thank the anonymous referee for helpful inputs which improved the exposition of this paper.

\section{Relative hyperbolicity \`a la Groves and Manning}
\label{basic_defs}

\begin{defn}[Groves and Manning \cite{groves_manning_horoball}]\label{defn_comb_horoball}
Let $\Lambda$ be a graph. The \defterm{combinatorial horoball based on $\Lambda$}, denoted by $\mathcal{H}\bigl(\Lambda\bigr)$, is a graph constructed as follows: 
\begin{itemize}
\item The vertex set is defined as
  $\mathcal{H}\bigl(\Lambda\bigr)^{(0)} := \Lambda^{(0)} \times
  \mathbb{N}_0$, where $\Lambda^{(0)}$ is the vertex set of $\Lambda$.
\item There are two kinds of edges in $\mathcal{H}\bigl(\Lambda\bigr)$: 

\begin{enumerate}
\item For each $n\in \mathbb{N}_0$ and each $v \in \Lambda^{(0)}$, there is a \emph{vertical} edge in $\mathcal{H}\bigl(\Lambda\bigr)$ between $(v,n)$ and $(v,n+1)$.

\item For each $n \in \mathbb{N}_0$, and each pair of vertices $(v,n)$ and $(w,n)$, there is a \emph{horizontal} edge between $(v,n)$ and $(w,n)$ if and only if $ 0 < d_{\Lambda}(v,w) \leq 2^n$.
\end{enumerate}
\end{itemize}
We denote by $\Lambda \times \{k\}$ the subgraph of
$\mathcal{H}(\Lambda)$ spanned by the vertex set
$\Lambda^{(0)} \times \{k\}$.
\end{defn}

\begin{defn}
  A \emph{rough isometry} is a quasi-isometry with multiplicative
    constant $1$.
\end{defn}

\begin{defn}
  Recall that, for each $k \in \N$, the \emph{Rips graph}
    $\Rips_{k}(\Lambda$) of a graph $\Lambda$ is the graph with vertex
    set $\Lambda^{(0)}$ and edges consisting of pairs of vertices at
    distance at most $k$ in $\Lambda$.
\end{defn}

\begin{rmk}
  \label{levels_rips} Observe that the bijection
  $\Lambda^{(0)} \xrightarrow{\cong} \Lambda^{(0)} \times \{n\}
  \subset \mathcal{H}(\Lambda)^{(0)}$ given by $v \mapsto (v,n)$
  extends to an isomorphism
  $\Rips_{2^n}(\Lambda) \xrightarrow{\cong} \Lambda \times \{n\}
  \subset \mathcal{H}(\Lambda)$.  In particular,
  $\Lambda \times \{0\}$ is isomorphic to $\Lambda$ and, for each $n$,
  the subgraph $\Lambda \times \{n\}$ is roughly isometric to
  $\Lambda$ with the metric scaled by $\frac{1}{2^n}$.
\end{rmk}

\begin{defn}[Groves and Manning \cite{groves_manning_horoball}]
  Let $\Gamma$ be a graph and
  $(\Lambda_{\alpha})_{\alpha \in \mathscr{A}}$ be a family of
  subgraphs of $\Gamma$. The \defterm{augmented space}
  $\mathcal{H}\bigl(\Gamma, (\Lambda_{\alpha})_{\alpha \in
    \mathscr{A}}\bigr)$ is the graph obtained by attaching, for each
  $\alpha \in \mathscr{A}$, the combinatorial horoball
  $\mathcal{H}\bigl(\Lambda_{\alpha}\bigr)$ to $\Gamma$ by identifying
  the subgraph $\Lambda_{\alpha} \subset \Gamma$ with the subgraph
  $\Lambda_{\alpha} \times \{0\} \subset
  \mathcal{H}\bigl(\Lambda_{\alpha}\bigr)$ along the isomorphism
  $\Lambda_{\alpha} \xrightarrow{\cong} \Lambda_{\alpha} \times \{0\}$
  given by $v \mapsto (v,0)$.
\end{defn}

\begin{defn}\label{defn_rel_hyp_graphs}
Let $\Gamma$ be a graph and $(\Lambda_{\alpha})_{\alpha \in \mathscr{A}}$ be a family of subgraphs of $\Gamma$. Then $\Gamma$ is \defterm{hyperbolic relative to $(\Lambda_{\alpha})_{\alpha \in \mathscr{A}}$} if the \defterm{augmented space} $\mathcal{H}\bigl(\Gamma, (\Lambda_{\alpha})_{\alpha \in \mathscr{A}}\bigr)$ is $\delta$-hyperbolic for some $\delta$. In that case, we call each $\Lambda_{\alpha} = \Lambda_{\alpha} \times \{0\}$ a \defterm{parabolic} subgraph of $\Gamma$.
\end{defn}

\begin{rmk}
The above definition for graphs is motivated by the characterization of relative hyperbolicity for groups by Groves and Manning (see \cref{defn_rel_hyp_groups} below). Our definition is likely equivalent to metric notions of relative hyperbolicity as investigated in \cite{sisto_metric_rh}, but we do not prove nor do we need such an equivalence for the purposes of this paper. \end{rmk}

\begin{defn}\label{defn_restricted_augmentation}
Let $\Gamma$ be a graph and $(\Lambda_{\alpha})_{\alpha \in \mathscr{A}}$ be a family of subgraphs of $\Gamma$. The \defterm{$n$-restricted augmentation} $\mathcal{H}_n\bigl(\Gamma, (\Lambda_{\alpha})_{\alpha \in \mathscr{A}}\bigr)$ is the subgraph of $\mathcal{H}\bigl(\Gamma, (\Lambda_{\alpha})_{\alpha \in \mathscr{A}}\bigr)$ spanned by the vertex set $\Gamma^{(0)} \sqcup \bigsqcup_{\alpha \in \mathscr{A}, k \in \{1, \cdots, n\}} \Lambda_{\alpha}^{(0)}\times \{k\}$. 

Similarly, the \defterm{$n$-restricted horoball} $\mathcal{H}_n\bigl(\Lambda\bigr)$ is the subgraph of the horoball $\mathcal{H}\bigl(\Lambda\bigr)$ spanned by the vertex set $\bigsqcup_{k \in \{1, \cdots, n\}} \Lambda \times \{k\}$.
\end{defn}

\begin{rmk}
  \label{action_on_n_restricted_horoball}
  If a group $G$ acts properly and cocompactly on $\Gamma$ and $(\Lambda_{\alpha})_{\alpha}$ is $G$-invariant then $G$ acts properly and cocompactly on $\mathcal{H}_n\bigl(\Gamma, (\Lambda_{\alpha})_{\alpha \in \mathscr{A}}\bigr)$.  Moreover, the embedding of $\Gamma$ in $\mathcal{H}_n\bigl(\Gamma, (\Lambda_{\alpha})_{\alpha \in \mathscr{A}}\bigr)$ is $G$-equivariant and, for any $\alpha$, the stabilizer of $\Lambda_{\alpha}\times \{n\}$ is equal to the stabilizer of $\Lambda_{\alpha}$.
\end{rmk}

\begin{rmk}\label{lem_relhyp_restricted_augmentation}
The graph $\Gamma$ is hyperbolic relative to $(\Lambda_{\alpha})_{\alpha \in \mathscr{A}}$ if and only if for each (any) $n \in \mathbb{N}_0$, $\mathcal{H}_n\bigl(\Gamma, (\Lambda_{\alpha})_{\alpha \in \mathscr{A}}\bigr)$ is hyperbolic relative to $(\Lambda_{\alpha}\times \{n\})_{\alpha \in \mathscr{A}}$.  Thus, when we speak of the \emph{parabolics} of $\mathcal{H}_n\bigl(\Gamma, (\Lambda_{\alpha})_{\alpha}\bigr)$ we will mean the top levels of the $n$-restricted horoballs $(\Lambda_{\alpha}\times \{n\})_{\alpha}$.
\end{rmk}

The following definition is due to Groves and Manning, who
  prove that it is equivalent to \emph{strong relative
    hyperbolicity} \cite{farb_rh,bowditch_rh}. We refer the reader to \cite[Theorem
  3.25]{groves_manning_horoball} for a proof and more details. A detailed study and equivalences of various notions of relative hyperbolicity was done by Hruska in \cite{hruska_rh}.

\begin{defn}\label{defn_rel_hyp_groups}
Let $G$ be a finitely generated group and let $H_1, \ldots, H_k$ be a family of finitely generated subgroups of $G$. For $1 \leq i \leq k$, let $S_i$ be a finite generating set for $H_i$ and let $S$ be a finite generating set for $G$ such that each $S_i \subset S$. 
Denote by $\Gamma$ the Cayley graph $\Cay(G,S)$ and, for $1 \leq i \leq k$, 
and $g \in G$, denote by $g\Lambda_i$ the subgraph of $\Gamma$ with vertex set $gH_i$ and edges labelled by $gS_i$.  
Then $G$ is \defterm{hyperbolic relative to $\{H_1, \cdots, H_k\}$} if $\Gamma$ is hyperbolic relative to 
$\{g\Lambda_i\}_{1 \leq i \leq k, g \in G}$.
\end{defn}

\section{Horoballs and convexity of parabolics} \label{sec:convexity}

It is well-known that given a relatively hyperbolic group, its parabolic subgroups are quasiconvex \cite[Lemma 4.15]{Drutu:2005}. The goal of this section is to prove
\cref{thm:convexify}, which says that a relatively hyperbolic graph
can be modified so that its parabolic subgraphs are convex subgraphs.
We make use of several previously known results.

\begin{lem}[{See Bridson and Haefliger
    \cite[Theorem~III.H.1.13]{Bridson:1999}}]
  \label{lem:locgeo} Let $\Gamma$ be a $\delta$-hyperbolic space and
  let $r > 8\delta + 1$.  Then there exists a constant
  $K = K(\delta,r)$ depending only on $\delta$ and $r$ such that the
  following holds.  If $\gamma$ is a path in $\Gamma$ and every
  subpath of length $r$ of $\gamma$ is a geodesic then $\gamma$ is a
  $(2\delta,K)$-quasi-geodesic.
\end{lem}

\begin{thrm}[{See Bridson and Haefliger \cite[Theorem~III.H.1.7]{Bridson:1999}}]
  \label{thm:morslem} Let $\Gamma$ be a $\delta$-hyperbolic graph.
  Let $L > 0$ and $K \ge 0$.  Then there exists a constant
  $M = M(\delta,L,K)$ such that for any two $(L,K)$-quasigeodesics
  $\beta_1$ and $\beta_2$ with the same endpoints, the images
  $\im(\beta_1)$ and $\im(\beta_2)$ are at Hausdorff distance at most
  $M$.
\end{thrm}

\begin{lem}\label{lem:georestriction}
Let $\mathcal{H}_n\bigl(\Lambda\bigr)$ be an $n$-restricted horoball. 
Let $v_1, v_2 \in \mathcal{H}_n\bigl(\Lambda\bigr)$ be given. The following hold:
\begin{enumerate}
\item There exists a geodesic $\beta$ between $v_1,v_2$ whose image consists of at most two vertical segments and one horizontal segment. If the horizontal segment is not contained in $\Lambda \times \{n\}$, then it is of length at most $3$. Further, any geodesic between the two points is at Hausdorff distance at most $4$ from $\im(\beta)$.

\item If the horizontal segment of $\im(\beta)$ is contained in $\Lambda\times \{K\}$, then the image of any geodesic between $v_1$ and $v_2$ is disjoint from $\Lambda\times \{K'\}$ for all $K' > K$.

\item Moreover, if $k$ is the least number such that either $v_1$ or $v_2$ is contained in $\Lambda\times \{k\}$, then the image of any geodesic between the points is contained in $\mathcal{H}_n\bigl(\Lambda\bigr) \setminus \mathcal{H}_{k-1}\bigl(\Lambda\bigr)$. 
\end{enumerate}
\end{lem}

\cref{lem:georestriction} is essentially a re-statement of Lemma~3.10 of \cite{groves_manning_horoball} in the context of restricted horoballs, and our proof below, given for the sake of completeness, is almost identical to theirs.

Let us first make the convention that a vertical segment of a path $\gamma$ is a subpath whose image is the union of vertical edges in a horoball. Similarly a horizontal segment is a subpath whose image is disjoint from the set of vertical edges.
\begin{proof}
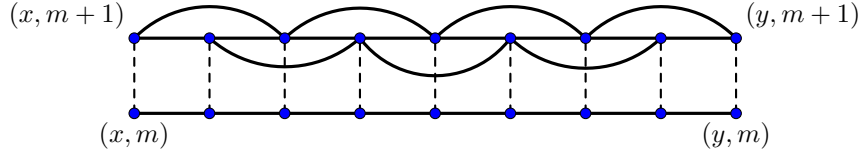
\begin{figure}
    \centering
    \begin{tikzpicture}[line cap=round,line join=round,>=triangle 45,x=1cm,y=1cm]
\draw [line width=1.2pt] (-4,4)-- (4,4);
\draw [line width=1.2pt] (-4,5)-- (4,5);
\draw [shift={(-3,4.02)},line width=1.2pt]  plot[domain=0.7752974968121261:2.3662951567776673,variable=\t]({1*1.4001428498549706*cos(\t r)+0*1.4001428498549706*sin(\t r)},{0*1.4001428498549706*cos(\t r)+1*1.4001428498549706*sin(\t r)});
\draw [shift={(-2,6.123684210526317)},line width=1.2pt]  plot[domain=3.9851654978244673:5.439612462944912,variable=\t]({1*1.5042161430413357*cos(\t r)+0*1.5042161430413357*sin(\t r)},{0*1.5042161430413357*cos(\t r)+1*1.5042161430413357*sin(\t r)});
\draw [shift={(-1,3.9505)},line width=1.2pt]  plot[domain=0.8095457010490638:2.3320469525407295,variable=\t]({1*1.4496379720468133*cos(\t r)+0*1.4496379720468133*sin(\t r)},{0*1.4496379720468133*cos(\t r)+1*1.4496379720468133*sin(\t r)});
\draw [shift={(1,4.02)},line width=1.2pt]  plot[domain=0.7752974968121267:2.3662951567776664,variable=\t]({1*1.4001428498549713*cos(\t r)+0*1.4001428498549713*sin(\t r)},{0*1.4001428498549713*cos(\t r)+1*1.4001428498549713*sin(\t r)});
\draw [shift={(3,4.02)},line width=1.2pt]  plot[domain=0.7752974968121261:2.3662951567776673,variable=\t]({1*1.4001428498549706*cos(\t r)+0*1.4001428498549706*sin(\t r)},{0*1.4001428498549706*cos(\t r)+1*1.4001428498549706*sin(\t r)});
\draw [shift={(0,5.75)},line width=1.2pt]  plot[domain=3.7850937623830774:5.639684198386302,variable=\t]({1*1.25*cos(\t r)+0*1.25*sin(\t r)},{0*1.25*cos(\t r)+1*1.25*sin(\t r)});
\draw [shift={(2,6.05)},line width=1.2pt]  plot[domain=3.9513762261599594:5.4734017346094195,variable=\t]({1*1.45*cos(\t r)+0*1.45*sin(\t r)},{0*1.45*cos(\t r)+1*1.45*sin(\t r)});
\draw [line width=0.8pt,dashed] (-4,4)-- (-4,5);
\draw [line width=0.8pt,dashed] (-3,4)-- (-3,5);
\draw [line width=0.8pt,dashed] (-2,4)-- (-2,5);
\draw [line width=0.8pt,dashed] (-1,4)-- (-1,5);
\draw [line width=0.8pt,dashed] (0,4)-- (0,5);
\draw [line width=0.8pt,dashed] (1,4)-- (1,5);
\draw [line width=0.8pt,dashed] (2,4)-- (2,5);
\draw [line width=0.8pt,dashed] (3,4)-- (3,5);
\draw [line width=0.8pt,dashed] (4,4)-- (4,5);

\draw [fill=blue] (-4,4) circle (2pt);

\draw [below] (-4,4) node {$(x,m)$};

\draw [fill=blue] (4,4) circle (2pt);

\draw [below] (4,4) node {$(y,m)$};

\draw [fill=blue] (-4,5) circle (2pt);

\draw [above left] (-4,5) node {$(x,m+1)$};

\draw [fill=blue] (4,5) circle (2pt);

\draw [above right] (4,5) node {$(y,m+1)$};

\draw [fill=blue] (-2,5) circle (2pt);
\draw [fill=blue] (-3,5) circle (2pt);
\draw [fill=blue] (-1,5) circle (2pt);
\draw [fill=blue] (0,5) circle (2pt);
\draw [fill=blue] (2,5) circle (2pt);
\draw [fill=blue] (1,5) circle (2pt);
\draw [fill=blue] (3,5) circle (2pt);
\draw [fill=blue] (-3,4) circle (2pt);
\draw [fill=blue] (-2,4) circle (2pt);
\draw [fill=blue] (-1,4) circle (2pt);
\draw [fill=blue] (0,4) circle (2pt);
\draw [fill=blue] (1,4) circle (2pt);
\draw [fill=blue] (2,4) circle (2pt);
\draw [fill=blue] (3,4) circle (2pt);

\end{tikzpicture}
    \caption{The ($\Lambda\times \{m\}$)-distance between $(x,m)$ and $(y,m)$ is $8$ while the ($\Lambda\times \{m+1\}$)-distance between $(x,m+1)$ and $(y,m+1)$ is $4$.}
    \label{fig:deeper_shorter}
\end{figure}
We start the proof with a basic observation. Let $1 \leq m < n$ and let $(x,m)$ and $(y,m)$ be two points in $\Lambda\times \{m\}$.
If $(x,m)$ and $(y,m)$ are at ($\Lambda\times \{m\}$)-distance $D$, note that the ($\Lambda\times \{m+1\}$)-distance between $(x,m+1)$ and $(y,m+1)$ is $\bigl\lceil\frac{D}{2}\bigr\rceil$ (see \cref{fig:deeper_shorter}). 
Similarly, the ($\Lambda\times \{m+k\}$)-distance between $(x,m+k)$ and $(y,m+k)$ is $\bigl\lceil\frac{D}{2^k}\bigr\rceil$. This observation implies the following:

\begin{enumerate}
\item Assume that a geodesic path contains a horizontal segment in $\Lambda\times \{m\}$ of length more than one. Assume that this horizontal segment is not contained in a strictly larger horizontal segment of the geodesic. Then the vertical segment immediately preceding the horizontal segment is an \emph{ascending} segment, in the sense that it is a vertical segment from some $\Lambda\times \{m-k\}$ to $\Lambda\times \{m\}$. Similarly, the immediate successor of the horizontal segment is a \emph{descending} segment. See \cref{fig:horizontal_maximal} for an illustration.

\item Any geodesic path with a descending segment at $(x,m)$ 
cannot ascend back to $\Lambda\times \{m\}$ in the future (see \cref{fig:descend_ascend}). In other words, no ascending segment follows a descending segment.

\item Any geodesic path contains at most two maximal descending (respectively ascending) segments.  See \cref{fig:descend_descend_descend}. 
\end{enumerate}

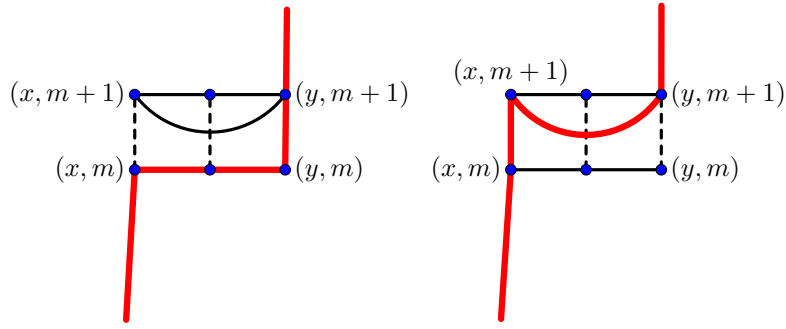
\begin{figure}
    \centering
    \begin{tikzpicture}[line cap=round,line join=round,>=triangle 45,x=1cm,y=1cm]
\draw [shift={(0,5.75)},line width=1.2pt]  plot[domain=3.7850937623830774:5.639684198386302,variable=\t]({1*1.25*cos(\t r)+0*1.25*sin(\t r)},{0*1.25*cos(\t r)+1*1.25*sin(\t r)});
\draw [line width=1.2pt,dashed] (-1,4)-- (-1,5);
\draw [line width=1.2pt,dashed] (0,4)-- (0,5);
\draw [line width=1.2pt,dashed] (1,4)-- (1,5);
\draw [line width=1.2pt] (-1,4)-- (1,4);
\draw [line width=1.2pt] (-1,5)-- (1,5);
\draw [line width=2.4pt,color=red] (-1.114565491450606,2.004012331929159)-- (-1,4);
\draw [line width=2.4pt,color=red] (-1,4)-- (1,4);
\draw [line width=2.4pt,color=red] (1,4)-- (1.0233591235908326,6.129538737516927);
\draw [line width=1.2pt] (4,4)-- (5,4);
\draw [line width=1.2pt] (5,4)-- (6,4);
\draw [line width=1.2pt] (4,5)-- (5,5);
\draw [line width=1.2pt] (5,5)-- (6,5);
\draw [line width=1.2pt,dashed] (4,4)-- (4,5);
\draw [line width=1.2pt,dashed] (5,4)-- (5,5);
\draw [line width=1.2pt,dashed] (6,4)-- (6,5);
\draw [line width=2.4pt,color=red] (3.8688890001423992,2.015632578651128)-- (4,4);
\draw [line width=2.4pt,color=red] (4,4)-- (4,5);
\draw [line width=2.4pt,color=red] (6,5)-- (6.002195633765171,6.181854945490882);
\draw [shift={(5,5.6618412789196295)},line width=2.4pt,color=red]  plot[domain=3.7262471600104305:5.698530800758949,variable=\t]({1*1.1991805028776814*cos(\t r)+0*1.1991805028776814*sin(\t r)},{0*1.1991805028776814*cos(\t r)+1*1.1991805028776814*sin(\t r)});
\draw [fill=blue] (-1,5) circle (2pt);
\draw [fill=blue] (0,5) circle (2pt);
\draw [fill=blue] (1,5) circle (2pt);
\draw [fill=blue] (-1,4) circle (2pt);
\draw [fill=blue] (0,4) circle (2pt);
\draw [fill=blue] (1,4) circle (2pt);
\draw [fill=blue] (4,4) circle (2pt);
\draw [fill=blue] (5,4) circle (2pt);
\draw [fill=blue] (6,4) circle (2pt);
\draw [fill=blue] (4,5) circle (2pt);
\draw [fill=blue] (5,5) circle (2pt);
\draw [fill=blue] (6,5) circle (2pt);

\draw [left] (-1,4) node {$(x,m)$};
\draw [right] (1,4) node {$(y,m)$};
\draw [left] (-1,5) node {$(x,m+1)$};
\draw [right] (1,5) node {$(y,m+1)$};

\draw [left] (4,4) node {$(x,m)$};
\draw [right] (6,4) node {$(y,m)$};
\draw [above] (4,5) node {$(x,m+1)$};
\draw [right] (6,5) node {$(y,m+1)$};

\end{tikzpicture}
    \caption{The red path between $(x,m)$ and $(y,m+1)$ on the left is longer than the red path on the right.}
    \label{fig:horizontal_maximal}
\end{figure}

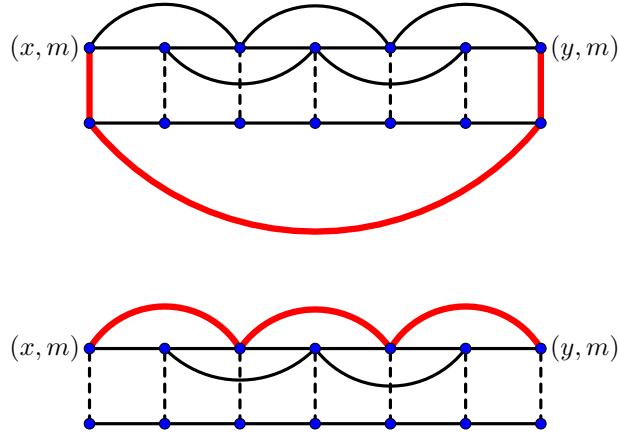
\begin{figure}
    \centering
    \begin{tikzpicture}[line cap=round,line join=round,>=triangle 45,x=1cm,y=1cm]

\draw [line width=1.2pt] (-3,4)-- (-2,4);
\draw [line width=1.2pt] (-2,4)-- (-1,4);
\draw [line width=1.2pt] (-1,4)-- (0,4);
\draw [line width=1.2pt] (0,4)-- (1,4);
\draw [line width=1.2pt] (1,4)-- (2,4);
\draw [line width=1.2pt] (2,4)-- (3,4);
\draw [line width=1.2pt] (-3,3)-- (-2,3);
\draw [line width=1.2pt] (-2,3)-- (-1,3);
\draw [line width=1.2pt] (-1,3)-- (0,3);
\draw [line width=1.2pt] (0,3)-- (1,3);
\draw [line width=1.2pt] (1,3)-- (2,3);
\draw [line width=1.2pt] (2,3)-- (3,3);
\draw [shift={(-2,3.4282758620689653)},line width=1.2pt]  plot[domain=0.5193688975489417:2.6222237560408517,variable=\t]({1*1.1518977775362642*cos(\t r)+0*1.1518977775362642*sin(\t r)},{0*1.1518977775362642*cos(\t r)+1*1.1518977775362642*sin(\t r)});
\draw [shift={(0,3.387142857142856)},line width=1.2pt]  plot[domain=0.5498196829613459:2.5917729706284476,variable=\t]({1*1.1728571428571435*cos(\t r)+0*1.1728571428571435*sin(\t r)},{0*1.1728571428571435*cos(\t r)+1*1.1728571428571435*sin(\t r)});
\draw [shift={(2,3.428275862068966)},line width=1.2pt]  plot[domain=0.5193688975489408:2.622223756040852,variable=\t]({1*1.1518977775362635*cos(\t r)+0*1.1518977775362635*sin(\t r)},{0*1.1518977775362635*cos(\t r)+1*1.1518977775362635*sin(\t r)});
\draw [shift={(-1,4.80125)},line width=1.2pt]  plot[domain=3.81709532631806:5.6076826344513195,variable=\t]({1*1.2814060880532756*cos(\t r)+0*1.2814060880532756*sin(\t r)},{0*1.2814060880532756*cos(\t r)+1*1.2814060880532756*sin(\t r)});
\draw [shift={(1,4.797916666666666)},line width=1.2pt]  plot[domain=3.8150619790012943:5.609715981768085,variable=\t]({1*1.2793244338104557*cos(\t r)+0*1.2793244338104557*sin(\t r)},{0*1.2793244338104557*cos(\t r)+1*1.2793244338104557*sin(\t r)});
\draw [line width=1.2pt,dashed] (-3,4)-- (-3,3);
\draw [line width=1.2pt,dashed] (-2,3)-- (-2,4);
\draw [line width=1.2pt,dashed] (-1,4)-- (-1,3);
\draw [line width=1.2pt,dashed] (0,3)-- (0,4);
\draw [line width=1.2pt,dashed] (1,4)-- (1,3);
\draw [line width=1.2pt,dashed] (2,3)-- (2,4);
\draw [line width=1.2pt,dashed] (3,4)-- (3,3);
\draw [line width=2.4pt,color=red] (-3,4)-- (-3,3);
\draw [shift={(0,5.39375)},line width=2.4pt,color=red]  plot[domain=3.8150619790012943:5.609715981768085,variable=\t]({1*3.8379733014313686*cos(\t r)+0*3.8379733014313686*sin(\t r)},{0*3.8379733014313686*cos(\t r)+1*3.8379733014313686*sin(\t r)});

 \draw [line width=2.4pt,color=red] (3,3)-- (3,4);

\draw [fill=blue] (-3,4) circle (2pt);
\draw [fill=blue] (-2,4) circle (2pt);
\draw [fill=blue] (-1,4) circle (2pt);
\draw [fill=blue] (0,4) circle (2pt);
\draw [fill=blue] (1,4) circle (2pt);
\draw [fill=blue] (2,4) circle (2pt);
\draw [fill=blue] (3,4) circle (2pt);
\draw [fill=blue] (-3,3) circle (2pt);
\draw [fill=blue] (-2,3) circle (2pt);
\draw [fill=blue] (-1,3) circle (2pt);
\draw [fill=blue] (0,3) circle (2pt);
\draw [fill=blue] (1,3) circle (2pt);
\draw [fill=blue] (2,3) circle (2pt);
\draw [fill=blue] (3,3) circle (2pt);

\draw[left] (-3,4) node {$(x,m)$};
\draw[right] (3,4) node {$(y,m)$};

\begin{scope}[xshift=-10cm, yshift=-4cm]

\draw [line width=1.2pt] (7,4)-- (13,4);
\draw [line width=1.2pt] (7,3)-- (13,3);
\draw [line width=1.2pt,dashed] (7,3)-- (7,4);
\draw [line width=1.2pt,dashed] (8,3)-- (8,4);
\draw [line width=1.2pt,dashed] (9,3)-- (9,4);
\draw [line width=1.2pt,dashed] (10,3)-- (10,4);
\draw [line width=1.2pt,dashed] (11,3)-- (11,4);
\draw [line width=1.2pt,dashed] (12,3)-- (12,4);
\draw [line width=1.2pt,dashed] (13,3)-- (13,4);
\draw [shift={(8,3.3875)},line width=1.2pt]  plot[domain=0.5495600135263804:2.5920326400634126,variable=\t]({1*1.1726705632870646*cos(\t r)+0*1.1726705632870646*sin(\t r)},{0*1.1726705632870646*cos(\t r)+1*1.1726705632870646*sin(\t r)});
\draw [shift={(10,3.298846153846151)},line width=1.2pt]  plot[domain=0.6114999380743488:2.5300927155154445,variable=\t]({1*1.2213176146999334*cos(\t r)+0*1.2213176146999334*sin(\t r)},{0*1.2213176146999334*cos(\t r)+1*1.2213176146999334*sin(\t r)});
\draw [shift={(12,3.3875)},line width=1.2pt]  plot[domain=0.5495600135263797:2.5920326400634135,variable=\t]({1*1.1726705632870642*cos(\t r)+0*1.1726705632870642*sin(\t r)},{0*1.1726705632870642*cos(\t r)+1*1.1726705632870642*sin(\t r)});
\draw [shift={(9,4.98047619047619)},line width=1.2pt]  plot[domain=3.9171329973404307:5.507644963428948,variable=\t]({1*1.40047619047619*cos(\t r)+0*1.40047619047619*sin(\t r)},{0*1.40047619047619*cos(\t r)+1*1.40047619047619*sin(\t r)});
\draw [shift={(11,4.75)},line width=1.2pt]  plot[domain=3.7850937623830774:5.639684198386302,variable=\t]({1*1.25*cos(\t r)+0*1.25*sin(\t r)},{0*1.25*cos(\t r)+1*1.25*sin(\t r)});
\draw [shift={(8,3.3875)},line width=2.4pt,color=red]  plot[domain=0.5495600135263773:2.5920326400634144,variable=\t]({1*1.1726705632870615*cos(\t r)+0*1.1726705632870615*sin(\t r)},{0*1.1726705632870615*cos(\t r)+1*1.1726705632870615*sin(\t r)});
\draw [shift={(10,3.2988461538461493)},line width=2.4pt,color=red]  plot[domain=0.61149993807435:2.530092715515443,variable=\t]({1*1.2213176146999345*cos(\t r)+0*1.2213176146999345*sin(\t r)},{0*1.2213176146999345*cos(\t r)+1*1.2213176146999345*sin(\t r)});
\draw [shift={(12,3.3875)},line width=2.4pt,color=red]  plot[domain=0.5495600135263811:2.592032640063412,variable=\t]({1*1.172670563287065*cos(\t r)+0*1.172670563287065*sin(\t r)},{0*1.172670563287065*cos(\t r)+1*1.172670563287065*sin(\t r)});

\draw [fill=blue] (7,4) circle (2pt);
\draw [fill=blue] (13,4) circle (2pt);
\draw [fill=blue] (7,3) circle (2pt);
\draw [fill=blue] (13,3) circle (2pt);
\draw [fill=blue] (8,3) circle (2pt);
\draw [fill=blue] (8,4) circle (2pt);
\draw [fill=blue] (9,3) circle (2pt);
\draw [fill=blue] (9,4) circle (2pt);
\draw [fill=blue] (10,3) circle (2pt);
\draw [fill=blue] (10,4) circle (2pt);
\draw [fill=blue] (11,3) circle (2pt);
\draw [fill=blue] (11,4) circle (2pt);
\draw [fill=blue] (12,3) circle (2pt);
\draw [fill=blue] (12,4) circle (2pt);

\draw[left] (7,4) node {$(x,m)$};
\draw[right] (13,4) node {$(y,m)$};
\end{scope}

\end{tikzpicture}
    \caption{The red path between $(x,m)$ and $(y,m)$ on the bottom panel is shorter than the one on the top panel.}
    \label{fig:descend_ascend}
\end{figure}

\begin{figure}
    \centering
    \begin{tikzpicture}[line cap=round,line join=round,>=triangle 45,x=1cm,y=1cm]

\draw [line width=1.2pt] (-6,2)-- (-2,2);
\draw [line width=1.2pt] (-6,1)-- (-2,1);
\draw [line width=1.2pt] (-6,3)-- (-2,3);
\draw [line width=1.2pt] (-6,0)-- (-2,0);
\draw [line width=1.2pt,dashed] (-6,3)-- (-6,0);
\draw [line width=1.2pt,dashed] (-5,3)-- (-5,0);
\draw [line width=1.2pt,dashed] (-4,0)-- (-4,3);
\draw [line width=1.2pt,dashed] (-3,3)-- (-3,0);
\draw [line width=1.2pt,dashed] (-2,3)-- (-2,0);
\draw [shift={(-5,2.53875)},line width=1.2pt]  plot[domain=0.4321699445935726:2.709422708996221,variable=\t]({1*1.10125*cos(\t r)+0*1.10125*sin(\t r)},{0*1.10125*cos(\t r)+1*1.10125*sin(\t r)});
\draw [shift={(-3,2.4279310344827594)},line width=1.2pt]  plot[domain=0.5196287396916754:2.6219639138981172,variable=\t]({1*1.1520689655172407*cos(\t r)+0*1.1520689655172407*sin(\t r)},{0*1.1520689655172407*cos(\t r)+1*1.1520689655172407*sin(\t r)});
\draw [shift={(-4,3.8016666666666663)},line width=1.2pt]  plot[domain=3.81734903007035:5.6074289306990295,variable=\t]({1*1.2816666666666665*cos(\t r)+0*1.2816666666666665*sin(\t r)},{0*1.2816666666666665*cos(\t r)+1*1.2816666666666665*sin(\t r)});
\draw [line width=2.4pt,color=red] (-6,3)-- (-6,2);
\draw [line width=2.4pt,color=red] (-6,2)-- (-5,2);
\draw [line width=2.4pt,color=red] (-5,2)-- (-5,1);
\draw [line width=2.4pt,color=red] (-5,1)-- (-4,1);
\draw [line width=2.4pt,color=red] (-4,1)-- (-4,0);

\draw [fill=blue] (-6,3) circle (2pt);
\draw [fill=blue] (-2,3) circle (2pt);
\draw [fill=blue] (-6,2) circle (2pt);
\draw [fill=blue] (-2,2) circle (2pt);
\draw [fill=blue] (-6,1) circle (2pt);
\draw [fill=blue] (-2,1) circle (2pt);
\draw [fill=blue] (-6,0) circle (2pt);
\draw [fill=blue] (-2,0) circle (2pt);
\draw [fill=blue] (-5,3) circle (2pt);
\draw [fill=blue] (-5,0) circle (2pt);
\draw [fill=blue] (-4,0) circle (2pt);
\draw [fill=blue] (-4,3) circle (2pt);
\draw [fill=blue] (-3,3) circle (2pt);
\draw [fill=blue] (-3,0) circle (2pt);
\draw [fill=blue] (-5,2) circle (2pt);
\draw [fill=blue] (-5,1) circle (2pt);
\draw [fill=blue] (-4,1) circle (2pt);
\draw [fill=blue] (-4,2) circle (2pt);
\draw [fill=blue] (-3,2) circle (2pt);
\draw [fill=blue] (-3,1) circle (2pt);

\draw[left] (-6,3) node {$(x,m)$};
\draw[below] (-4,0) node {$(y,m-3)$};

\begin{scope}[xshift=-3cm]

\draw [line width=1.2pt] (2,3)-- (6,3);
\draw [line width=1.2pt] (2,2)-- (6,2);
\draw [line width=1.2pt] (2,1)-- (6,1);
\draw [line width=1.2pt] (2,0)-- (6,0);
\draw [line width=1.2pt,dashed] (2,0)-- (2,3);
\draw [line width=1.2pt,dashed] (3,0)-- (3,3);
\draw [line width=1.2pt,dashed] (4,3)-- (4,0);
\draw [line width=1.2pt,dashed] (5,0)-- (5,3);
\draw [line width=1.2pt,dashed] (6,3)-- (6,0);
\draw [line width=2.4pt,color=red] (4,3)-- (4,0);
\draw [shift={(3,2.3871428571428575)},line width=1.2pt]  plot[domain=0.5498196829613449:2.5917729706284485,variable=\t]({1*1.1728571428571428*cos(\t r)+0*1.1728571428571428*sin(\t r)},{0*1.1728571428571428*cos(\t r)+1*1.1728571428571428*sin(\t r)});
\draw [shift={(5,2.4279310344827594)},line width=1.2pt]  plot[domain=0.5196287396916757:2.6219639138981177,variable=\t]({1*1.152068965517241*cos(\t r)+0*1.152068965517241*sin(\t r)},{0*1.152068965517241*cos(\t r)+1*1.152068965517241*sin(\t r)});
\draw [shift={(4,3.98047619047619)},line width=1.2pt]  plot[domain=3.9171329973404307:5.507644963428948,variable=\t]({1*1.40047619047619*cos(\t r)+0*1.40047619047619*sin(\t r)},{0*1.40047619047619*cos(\t r)+1*1.40047619047619*sin(\t r)});
\draw [shift={(3,2.3871428571428575)},line width=2.4pt,color=red]  plot[domain=0.5498196829613449:2.5917729706284485,variable=\t]({1*1.1728571428571428*cos(\t r)+0*1.1728571428571428*sin(\t r)},{0*1.1728571428571428*cos(\t r)+1*1.1728571428571428*sin(\t r)});

\draw [fill=blue] (2,3) circle (2pt);
\draw [fill=blue] (6,3) circle (2pt);
\draw [fill=blue] (3,3) circle (2pt);
\draw [fill=blue] (4,3) circle (2pt);
\draw [fill=blue] (5,3) circle (2pt);
\draw [fill=blue] (2,2) circle (2pt);
\draw [fill=blue] (6,2) circle (2pt);
\draw [fill=blue] (2,1) circle (2pt);
\draw [fill=blue] (6,1) circle (2pt);
\draw [fill=blue] (2,0) circle (2pt);
\draw [fill=blue] (6,0) circle (2pt);
\draw [fill=blue] (3,0) circle (2pt);
\draw [fill=blue] (4,0) circle (2pt);
\draw [fill=blue] (5,0) circle (2pt);
\draw [fill=blue] (3,2) circle (2pt);
\draw [fill=blue] (4,2) circle (2pt);
\draw [fill=blue] (5,2) circle (2pt);
\draw [fill=blue] (5,1) circle (2pt);
\draw [fill=blue] (4,1) circle (2pt);
\draw [fill=blue] (3,1) circle (2pt);

\draw[below] (2,3) node {$(x,m)$};
\draw[below] (4,0) node {$(y,m-3)$};

\end{scope}

\end{tikzpicture}
    \caption{The red path between $(x,m)$ and $(y,m-3)$ on the left is longer than the one on the right.}
    \label{fig:descend_descend_descend}
\end{figure}
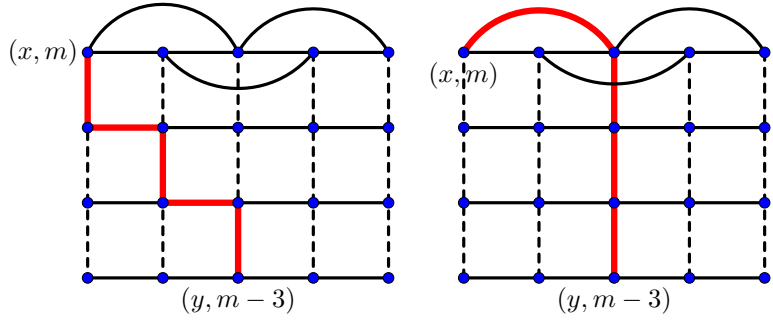

  Let $\gamma$ be a geodesic between the points $v_1$ and $v_2$ in the
  statement. By the above observations, if $\gamma$ contains a horizontal segment of length
  at least two at some $\Lambda\times \{m\}$, then $\im(\gamma)$ is disjoint
  from $\Lambda\times \{m'\}$ for all $m < m' \leq n$. 
Thus, any horizontal segment in $\gamma$ is either of length one, or is contained in the maximum level $\Lambda\times \{\max\}$ that intersects $\im(\gamma)$ nontrivially.

In fact, it can be verified
  that apart from the horizontal segment at $\Lambda\times \{\max\}$, the
  image of $\gamma$ can have at most one more horizontal edge.

Another consequence of the above observations is that if $\gamma$ contains a horizontal segment of length at least $6$, then this segment has to be contained in $\Lambda\times \{n\}$, see \cref{fig:longhorizontal}. 

\begin{figure}
    \centering
    \begin{tikzpicture}[line cap=round,line join=round,>=triangle 45,x=1cm,y=1cm]

\draw [line width=1.2pt] (-4,3)-- (2,3);
\draw [line width=1.2pt] (-4,2)-- (2,2);
\draw [line width=1.2pt,dashed] (-4,3)-- (-4,2);
\draw [line width=1.2pt,dashed] (-3,3)-- (-3,2);
\draw [line width=1.2pt,dashed] (-2,2)-- (-2,3);
\draw [line width=1.2pt,dashed] (-1,3)-- (-1,2);
\draw [line width=1.2pt,dashed] (0,2)-- (0,3);
\draw [line width=1.2pt,dashed] (1,3)-- (1,2);
\draw [line width=1.2pt,dashed] (2,2)-- (2,3);
\draw [shift={(-3,2.344444444444445)},line width=1.2pt]  plot[domain=0.580270798859531:2.561321854730262,variable=\t]({1*1.1957228300989122*cos(\t r)+0*1.1957228300989122*sin(\t r)},{0*1.1957228300989122*cos(\t r)+1*1.1957228300989122*sin(\t r)});
\draw [shift={(-1,2.5736363636363637)},line width=1.2pt]  plot[domain=0.4030250925051079:2.738567561084685,variable=\t]({1*1.0870997886179645*cos(\t r)+0*1.0870997886179645*sin(\t r)},{0*1.0870997886179645*cos(\t r)+1*1.0870997886179645*sin(\t r)});
\draw [shift={(1,2.5035483870967745)},line width=1.2pt]  plot[domain=0.4608048721177228:2.6807877814720706,variable=\t]({1*1.1164516129032256*cos(\t r)+0*1.1164516129032256*sin(\t r)},{0*1.1164516129032256*cos(\t r)+1*1.1164516129032256*sin(\t r)});
\draw [shift={(-2,3.9804761904761907)},line width=1.2pt]  plot[domain=3.917132997340431:5.507644963428948,variable=\t]({1*1.4004761904761907*cos(\t r)+0*1.4004761904761907*sin(\t r)},{0*1.4004761904761907*cos(\t r)+1*1.4004761904761907*sin(\t r)});
\draw [shift={(0,3.9804761904761907)},line width=1.2pt]  plot[domain=3.917132997340431:5.507644963428948,variable=\t]({1*1.4004761904761907*cos(\t r)+0*1.4004761904761907*sin(\t r)},{0*1.4004761904761907*cos(\t r)+1*1.4004761904761907*sin(\t r)});
\draw [line width=2.4pt,color=red] (-4,2)-- (2,2);

\draw [fill=blue] (-4,3) circle (2pt);
\draw [fill=blue] (2,3) circle (2pt);
\draw [fill=blue] (-4,2) circle (2pt);
\draw [fill=blue] (2,2) circle (2pt);
\draw [fill=blue] (-3,3) circle (2pt);
\draw [fill=blue] (-3,2) circle (2pt);
\draw [fill=blue] (-2,2) circle (2pt);
\draw [fill=blue] (-2,3) circle (2pt);
\draw [fill=blue] (-1,3) circle (2pt);
\draw [fill=blue] (-1,2) circle (2pt);
\draw [fill=blue] (0,2) circle (2pt);
\draw [fill=blue] (0,3) circle (2pt);
\draw [fill=blue] (1,3) circle (2pt);
\draw [fill=blue] (1,2) circle (2pt);

\draw[left] (-4,2) node {$(x,m)$};
\draw[right] (2,2) node {$(y,m)$};

\begin{scope}[yshift=1.5cm]
\draw [line width=1.2pt] (-4,-1)-- (2,-1);
\draw [line width=1.2pt] (-4,-2)-- (2,-2);
\draw [line width=1.2pt,dashed] (-4,-2)-- (-4,-1);
\draw [line width=1.2pt,dashed] (-3,-1)-- (-3,-2);
\draw [line width=1.2pt,dashed] (-2,-2)-- (-2,-1);
\draw [line width=1.2pt,dashed] (-1,-1)-- (-1,-2);
\draw [line width=1.2pt,dashed] (0,-2)-- (0,-1);
\draw [line width=1.2pt,dashed] (1,-1)-- (1,-2);
\draw [line width=1.2pt,dashed] (2,-2)-- (2,-1);
\draw [shift={(-2,-0.08363636363636352)},line width=1.2pt]  plot[domain=3.8833752312151177:5.541402729554261,variable=\t]({1*1.3563636363636364*cos(\t r)+0*1.3563636363636364*sin(\t r)},{0*1.3563636363636364*cos(\t r)+1*1.3563636363636364*sin(\t r)});
\draw [shift={(0,-0.019523809523809228)},line width=1.2pt]  plot[domain=3.917132997340431:5.507644963428948,variable=\t]({1*1.4004761904761907*cos(\t r)+0*1.4004761904761907*sin(\t r)},{0*1.4004761904761907*cos(\t r)+1*1.4004761904761907*sin(\t r)});
\draw [shift={(-3,-1.655925925925926)},line width=2.4pt,color=red]  plot[domain=0.5805297998580884:2.5610628537317046,variable=\t]({1*1.195925925925926*cos(\t r)+0*1.195925925925926*sin(\t r)},{0*1.195925925925926*cos(\t r)+1*1.195925925925926*sin(\t r)});
\draw [shift={(-1,-1.6544444444444444)},line width=2.4pt,color=red]  plot[domain=0.5794932679365729:2.5620993856532204,variable=\t]({1*1.195114024210325*cos(\t r)+0*1.195114024210325*sin(\t r)},{0*1.195114024210325*cos(\t r)+1*1.195114024210325*sin(\t r)});
\draw [shift={(1,-1.6128571428571428)},line width=2.4pt,color=red]  plot[domain=0.549819682961345:2.5917729706284485,variable=\t]({1*1.1728571428571428*cos(\t r)+0*1.1728571428571428*sin(\t r)},{0*1.1728571428571428*cos(\t r)+1*1.1728571428571428*sin(\t r)});
\draw [line width=2.4pt,color=red] (-4,-1)-- (-4,-2);
\draw [line width=2.4pt,color=red] (2,-1)-- (2,-2);
\draw [fill=blue] (-4,-1) circle (2pt);
\draw [fill=blue] (2,-1) circle (2pt);
\draw [fill=blue] (-4,-2) circle (2pt);
\draw [fill=blue] (2,-2) circle (2pt);
\draw [fill=blue] (-3,-1) circle (2pt);
\draw [fill=blue] (-3,-2) circle (2pt);
\draw [fill=blue] (-2,-2) circle (2pt);
\draw [fill=blue] (-2,-1) circle (2pt);
\draw [fill=blue] (-1,-1) circle (2pt);
\draw [fill=blue] (-1,-2) circle (2pt);
\draw [fill=blue] (0,-2) circle (2pt);
\draw [fill=blue] (0,-1) circle (2pt);
\draw [fill=blue] (1,-1) circle (2pt);
\draw [fill=blue] (1,-2) circle (2pt);

\draw[left] (-4,-2) node {$(x,m)$};
\draw[right] (2,-2) node {$(y,m)$};

\end{scope}

\end{tikzpicture}
    \caption{If $m <n$, then the red horizontal path between $(x,m)$ and $(y,m)$ in the top panel is longer than the red path in the bottom panel.}
    \label{fig:longhorizontal}
\end{figure}
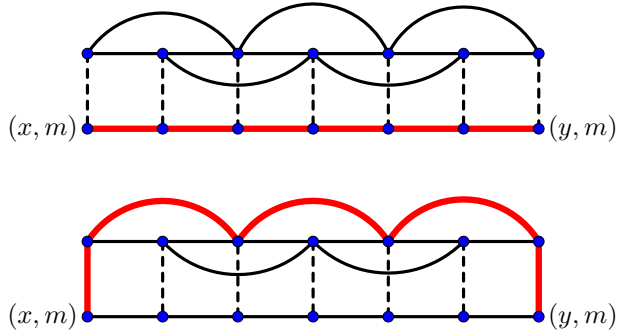

Assume that the horizontal edge not at $\Lambda \times \{\max\}$ is an edge between $(x,m)$ and $(y,m)$ and is followed by an ascending segment from $(y,m)$ to $(y,\max) \in \Lambda\times \{\max\}$. Let $\gamma'$ be the geodesic obtained from $\gamma$ by replacing the above by a vertical segment from $(x,m)$ to $(x,\max)$ followed by a horizontal edge to $(y,\max)$. 
If $\max < n$ and the only horizontal
segment of $\gamma'$ contains $4$ or $5$ edges, then let $\beta$ be
the geodesic obtained by replacing this horizontal segment by an ascending edge, a horizontal segment in $\Lambda\times \{\max+1\}$ and a descending edge back to $\Lambda\times \{\max\}$, similar to the procedure in \cref{fig:longhorizontal}. We leave it as an exercise to verify that $\beta$ is
as required.
\end{proof}

Before stating the main result of this section, we recall a convexity result from \cite{groves_manning_horoball} which will be used in the proof.

\begin{lem}[Lemma 3.26, \cite{groves_manning_horoball}] \label{lem:deephoroball_convexity}
 Let $\Gamma$ be a graph that is hyperbolic
  relative to a family $(\Lambda_{\alpha})_{\alpha \in \mathscr{A}}$ of subgraphs. Let $\delta$ be the hyperbolicity constant of $\mathcal{H}\bigl(\Gamma, (\Lambda_{\alpha})_{\alpha\in \mathscr{A}}\bigr)$. Then for any $k > \delta$ and any $\alpha \in \mathscr{A}$, $\mathcal{H}\bigl(\Lambda_{\alpha}\bigr) \setminus \mathcal{H}_k\bigl(\Lambda_{\alpha}\bigr)$ is convex in $\mathcal{H}\bigl(\Gamma, (\Lambda_{\alpha})_{\alpha\in \mathscr{A}}\bigr)$.
\end{lem}

\begin{thrm}
  \label{thm:convexify} Let $\Gamma$ be a graph that is hyperbolic
  relative to a family $(\Lambda_{\alpha})_{\alpha \in \mathscr{A}}$
  of subgraphs.  Then, for $n$ large enough, the parabolics (i.e. the
  top levels) of the restricted horoballs
  $\mathcal{H}_n\bigl(\Gamma, (\Lambda_{\alpha})_{\alpha\in
    \mathscr{A}}\bigr)$ are convex subgraphs.
\end{thrm}

\begin{proof}
Let $\delta$ be the hyperbolicity constant of $\mathcal{H}\bigl(\Gamma, (\Lambda_{\alpha})_{\alpha \in \mathscr{A}}\bigr)$. Let $r = \lceil 8\delta+2 \rceil$ and $n \ge 2r + M(\delta, 2\delta, K)$, where $K$ is the constant from \cref{lem:locgeo} and $M$ is the constant from \cref{thm:morslem}. Fix $\alpha_0 \in \mathscr{A}$ and points $x,y \in \Lambda_{\alpha_0}\times\{n\}$. Let $\gamma : P \to \mathcal{H}_n \bigl(\Gamma, (\Lambda_{\alpha})_{\alpha \in \mathscr{A}}\bigr)$ be a geodesic (in $\mathcal{H}_n \bigl(\Gamma, (\Lambda_{\alpha})_{\alpha \in \mathscr{A}}\bigr)$) between $x$ and $y$. Since each $n$-restricted horoball in $\mathcal{H}_n \bigl(\Gamma, (\Lambda_{\alpha})_{\alpha \in \mathscr{A}}\bigr)$ is a full subgraph, every subpath of $\gamma$ whose image lies in an $n$-restricted horoball is a geodesic in that horoball. We will therefore assume that each such geodesic subpath of $\gamma$ is of the form given by \cref{lem:georestriction}.

Denote by $U \subset \mathcal{H}_n\bigl(\Gamma, (\Lambda_{\alpha})_{\alpha \in \mathscr{A}}\bigr)$ the set $\bigcup_{\alpha \in \mathscr{A}} N_r \bigl(\Lambda_{\alpha}\times \{n\}\bigr)$. The path $\gamma$ is a concatenation $\gamma_1 \cdot \beta_1\cdot \gamma_2\cdots \gamma_k$, where each $\gamma_i$ is a path with image in $U$ and each $\beta_i$ is such that its image is disjoint from $U$, except at the endpoints. See \cref{fig:gamma_concatenation} for an illustration.
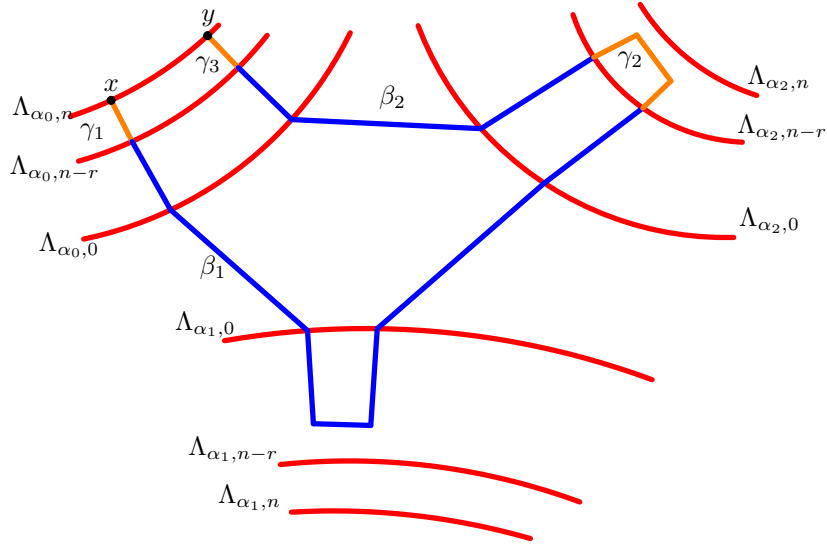
\begin{figure}
    \centering
    \begin{tikzpicture}[scale=.3,line cap=round,line join=round,>=triangle 45,x=1cm,y=1cm]

\draw [shift={(-18.45386621876636,21.396351664934475)},line width=2pt,color=red]  plot[domain=5.0287545651388985:5.500775668129845,variable=\t]({1*16.51258029884188*cos(\t r)+0*16.51258029884188*sin(\t r)},{0*16.51258029884188*cos(\t r)+1*16.51258029884188*sin(\t r)});
\draw [shift={(-17.332594835870733,19.334735126023688)},line width=2pt,color=red]  plot[domain=4.9861224626648895:5.616224426538524,variable=\t]({1*16.2187787822777*cos(\t r)+0*16.2187787822777*sin(\t r)},{0*16.2187787822777*cos(\t r)+1*16.2187787822777*sin(\t r)});
\draw [shift={(-16.167525878940722,16.915023832337)},line width=2pt,color=red]  plot[domain=4.916449599755689:5.8264240156174685,variable=\t]({1*17.004893851546004*cos(\t r)+0*17.004893851546004*sin(\t r)},{0*17.004893851546004*cos(\t r)+1*17.004893851546004*sin(\t r)});
\draw [shift={(-0.22161173602218817,-40.33080191208166)},line width=2pt,color=red]  plot[domain=1.2170161960980224:1.7427175995453863,variable=\t]({1*36.62848574858687*cos(\t r)+0*36.62848574858687*sin(\t r)},{0*36.62848574858687*cos(\t r)+1*36.62848574858687*sin(\t r)});
\draw [shift={(-1.6406139640687234,-43.39978841231607)},line width=2pt,color=red]  plot[domain=1.2914643682176212:1.630376342706491,variable=\t]({1*31.610137335569906*cos(\t r)+0*31.610137335569906*sin(\t r)},{0*31.610137335569906*cos(\t r)+1*31.610137335569906*sin(\t r)});
\draw [shift={(-0.9394409606742639,-39.0750620039208)},line width=2pt,color=red]  plot[domain=1.2177112390093088:1.6751714254869638,variable=\t]({1*29.50184578065782*cos(\t r)+0*29.50184578065782*sin(\t r)},{0*29.50184578065782*cos(\t r)+1*29.50184578065782*sin(\t r)});
\draw [shift={(15.723645079667886,14.91142804715737)},line width=2pt,color=red]  plot[domain=3.5051625730516993:4.737402875910497,variable=\t]({1*14.580188905473708*cos(\t r)+0*14.580188905473708*sin(\t r)},{0*14.580188905473708*cos(\t r)+1*14.580188905473708*sin(\t r)});
\draw [shift={(20.01907474396503,15.688462674180977)},line width=2pt,color=red]  plot[domain=3.69747728528241:4.4011780163230725,variable=\t]({1*9.51136003695752*cos(\t r)+0*9.51136003695752*sin(\t r)},{0*9.51136003695752*cos(\t r)+1*9.51136003695752*sin(\t r)});
\draw [shift={(16.860055582834992,12.890117941747803)},line width=2pt,color=red]  plot[domain=3.467905230591887:4.66505511258615,variable=\t]({1*8.33971916107538*cos(\t r)+0*8.33971916107538*sin(\t r)},{0*8.33971916107538*cos(\t r)+1*8.33971916107538*sin(\t r)});
\draw [line width=2pt,color=orange] (-11.504996923245754,6.41708186832984)-- (-10.575328228177213,4.590647050425002);
\draw [line width=2pt,color=blue] (-10.575328228177213,4.590647050425002)-- (-8.86859603710971,1.55624660020111);
\draw [line width=2pt,color=blue] (-8.86859603710971,1.55624660020111)-- (-2.8023117948805947,-3.7933422926631337);
\draw [line width=2pt,color=blue] (-2.8023117948805947,-3.7933422926631337)-- (-2.5422750726703884,-7.923509615478263);
\draw [line width=2pt,color=blue] (-2.5422750726703884,-7.923509615478263)-- (0,-8);
\draw [line width=2pt,color=blue] (0,-8)-- (0.287216371459457,-3.7058505511382833);
\draw [line width=2pt,color=blue] (0.287216371459457,-3.7058505511382833)-- (7.687732285975587,2.7456558114350926);
\draw [line width=2pt,color=orange] (13.297497328232488,7.275129140626043)-- (11.788947575765548,9.311671306456397);
\draw [line width=2pt,color=blue] (4.882191337236563,5.162336332027163)-- (-3.5044493818078433,5.5654463719154865);
\draw [line width=2pt,color=blue] (7.687732285975587,2.7456558114350926)-- (12.053028177723974,6.075178234595622);
\draw [line width=2pt,color=orange] (12.053028177723974,6.075178234595622)-- (13.297497328232488,7.275129140626043);
\draw [line width=2pt,color=blue] (4.882191337236563,5.162336332027163)-- (9.880043024690925,8.326091169942126);
\draw [line width=2pt,color=orange] (9.880043024690925,8.326091169942126)-- (11.788947575765548,9.311671306456397);
\draw [line width=2pt,color=orange] (-7.227881593929354,9.286744845197342)-- (-5.865269295274613,7.865243849840024);
\draw [line width=2pt,color=blue] (-5.865269295274613,7.865243849840024)-- (-3.5044493818078433,5.5654463719154865);

\draw [fill=black] (-11.504996923245754,6.41708186832984) circle (5pt);
\draw[above] (-11.5,6.4) node {$x$};
\draw[below left] (-11.329577380790317,5.785436260064951) node {$\gamma_1$};
\draw [fill=black] (-7.227881593929354,9.286744845197342) circle (5pt);
\draw[above] (-7.2,9.3) node {$y$};
\draw[left] (-5.974225759532677,-1.0030376260362281) node {$\beta_1$};
\draw[left] (12.5,8.2) node {$\gamma_2$};
\draw (0.9651031018152485,6.539711136298415) node {$\beta_2$};
\draw (-7.2,8) node {$\gamma_3$};
\draw (-14.5,6) node {$\Lambda_{\alpha_0,n}$};
\draw[below] (-14.007253191419137,4.352313995221368) node {$\Lambda_{\alpha_0,n-r}$};
\draw[below] (-13.366119546620688,0.958077052170779) node {$\Lambda_{\alpha_0,0}$};
\draw (-7.331920536752924,-3.5298584614183337) node {$\Lambda_{\alpha_1,0}$};
\draw (-6.087366990967698,-9.073778801734296) node {$\Lambda_{\alpha_1,n-r}$};
\draw (-5.370805858545902,-11.22346219899967) node {$\Lambda_{\alpha_1,n}$};
\draw (17.672291610386615,1.0712182836057986) node {$\Lambda_{\alpha_2,0}$};
\draw (18.23799776756172,5.182016359078179) node {$\Lambda_{\alpha_2,n-r}$};
\draw (18.087142792315024,7.3694135001552254) node {$\Lambda_{\alpha_2,n}$};

\end{tikzpicture}
    \caption{The path $\gamma$ is a concatenation of paths $\gamma_i$ (in orange) and $\beta_j$ (in blue).}
    \label{fig:gamma_concatenation}
\end{figure}

Note that by \cref{lem:georestriction}, each $\beta_i$ is a path which satisfies the following:
\begin{itemize}
    \item $\im(\beta_i)$ is not contained in any single $n$-restricted horoball and thus has length at least $2(n-r) > 2r$, and 
    \item for any $\alpha \in \mathscr{A}$, $\im(\beta_i) \cap \mathcal{H}_n\bigl(\Lambda_{\alpha}\bigr)$ is a union of components, where each component is either a vertical segment between $\Lambda_{\alpha}\times \{0\}$ and $\Lambda_{\alpha}\times \{n-r\}$ (e.g., $\im(\beta_2) \cap \mathcal{H}_n\bigl(\Lambda_{\alpha_2}\bigr)$ in \cref{fig:gamma_concatenation}), or the image of a geodesic between points of $\Lambda_{\alpha}\times \{0\}$ (e.g., $\im(\beta_1) \cap \mathcal{H}_n\bigl(\Lambda_{\alpha_1}\bigr)$ in \cref{fig:gamma_concatenation}). In the latter case, we note that this component is disjoint from the image of any $\gamma_j$.
\end{itemize}

Let $\iota : \mathcal{H}_n\bigl(\Gamma, (\Lambda_{\alpha})_{\alpha \in \mathscr{A}}\bigr) \hookrightarrow \mathcal{H}\bigl(\Gamma, (\Lambda_{\alpha})_{\alpha \in \mathscr{A}}\bigr)$ denote the inclusion map.
For each $i$, let $\gamma'_i$ be a geodesic path in $\mathcal{H}\bigl(\Gamma, (\Lambda_{\alpha})_{\alpha \in \mathscr{A}}\bigr)$ between the endpoints of $\iota \circ \gamma_i$. Let $\gamma' : P' \to \mathcal{H}\bigl(\Gamma, (\Lambda_{\alpha})_{\alpha \in \mathscr{A}}\bigr)$ be the path obtained from $\iota \circ \gamma$ by replacing each $\iota \circ \gamma_i$ by $\gamma'_i$. We will denote $\iota \circ \beta_i$ by $\beta'_i$. Thus $\gamma' = \gamma'_1 \cdot \beta'_1 \cdot \gamma'_2 \cdots \gamma'_k$.

\begin{claim*}
The path $\gamma'$ is an $r$-local geodesic in $\mathcal{H}\bigl(\Gamma, (\Lambda_{\alpha})_{\alpha \in \mathscr{A}}\bigr)$.
\end{claim*}
\begin{subproof}[Proof of claim.] Each $\gamma'_i$ is a geodesic, and therefore a local geodesic. Each $\beta'_i$ is an $r$-local geodesic since the $r$-ball around any point in $\im(\beta_i)$ is contained in $\mathcal{H}_n\bigl(\Gamma, (\Lambda_{\alpha})_{\alpha \in \mathscr{A}}\bigr)$. 

As observed above, the image of every subpath of $\beta'_i$ that lies in a horoball is either a vertical segment or it does not meet any $\gamma'_j$.
This implies that any subpath of $\beta'_{i-1}\cdot\gamma'_i\cdot\beta'_i$ whose image lies in $\mathcal{H}\bigl(\Lambda_{\alpha}\bigr) \setminus \mathcal{H}_r\bigl(\Lambda_{\alpha}\bigr)$ is a geodesic in $\mathcal{H}\bigl(\Lambda_{\alpha}\bigr) \setminus \mathcal{H}_r\bigl(\Lambda_{\alpha}\bigr)$. Since $\mathcal{H}\bigl(\Lambda_{\alpha}\bigr) \setminus \mathcal{H}_r\bigl(\Lambda_{\alpha}\bigr)$ is convex (by \cref{lem:deephoroball_convexity}), each such subpath is in fact a geodesic in $\mathcal{H}\bigl(\Gamma, (\Lambda_{\alpha})_{\alpha \in \mathscr{A}}\bigr)$, and therefore an $r$-local geodesic. This proves the claim.
\end{subproof}
Thus by \cref{lem:locgeo}, $\gamma'$ is a $(2\delta,K)$-quasi-geodesic and by \cref{thm:morslem}, it lies in an $M = M(\delta, 2\delta, K)$ neighborhood of any geodesic in $\mathcal{H}\bigl(\Gamma, (\Lambda_{\alpha})_{\alpha \in \mathscr{A}}\bigr)$ between $x$ and $y$. Since $x,y \in \Lambda_{\alpha_0}\times \{n\}$ with $n - 1 > \delta$, we have that any geodesic between them in $\mathcal{H}\bigl(\Gamma, (\Lambda_{\alpha})_{\alpha \in \mathscr{A}}\bigr)$ lies in $\mathcal{H}\bigl(\Lambda_{\alpha_0}\bigr)\setminus \mathcal{H}_{n-1}\bigl(\Lambda_{\alpha_0}\bigr)$ (again, by \cref{lem:deephoroball_convexity}).
This implies that $\gamma'$ lies in $N_M(\mathcal{H}\bigl(\Lambda_{\alpha_0}\bigr) \setminus \mathcal{H}_{n-1}\bigl(\Lambda_{\alpha_0}\bigr)) \subset \mathcal{H}\bigl(\Lambda_{\alpha_0}\bigr) \setminus \mathcal{H}_{2r}\bigl(\Lambda_{\alpha_0}\bigr)$.

We are thus forced to conclude that $\gamma' = \gamma'_1$ (and therefore $\gamma = \gamma_1$). Indeed, if not, then $\beta_1$ is a geodesic in $\mathcal{H}_n\bigl(\Gamma, (\Lambda_{\alpha})_{\alpha \in \mathscr{A}}\bigr)$ with endpoints on $\Lambda_{\alpha_0}\times\{n-r\}$ and such that $\im(\beta_1) \subset \mathcal{H}_n\bigl(\Lambda_{\alpha_0}\bigr)$. But as observed above, $\im(\beta_1)$ is not contained in any single $n$-restricted horoball, which is a contradiction.

Using \cref{lem:georestriction} once again, 
we conclude that $\gamma \subset \Lambda_{\alpha_0}\times \{n\}$.
\end{proof}

\begin{cor} \label{cor:rough_isometry_of_horopheres}
Let $\Gamma$ be a graph that is hyperbolic
  relative to a family $(\Lambda_{\alpha})_{\alpha \in \mathscr{A}}$ of subgraphs. Let $n$ be such that the parabolics $(\Lambda_{\alpha}\times \{n\})_{\alpha}$ of
  $\mathcal{H}_n\bigl(\Gamma, (\Lambda_{\alpha})_{\alpha\in \mathscr{A}}\bigr)$ are
  convex subgraphs, as in \cref{thm:convexify}. Then for each $\alpha \in \mathscr{A}$, the subspace $\Lambda_{\alpha}^{(0)}\times \{0\}$ is roughly isometric to the subgraph $\Lambda_{\alpha}\times \{n\}$.
\end{cor}

\begin{proof}
Let $(x,0), (y,0) \in \Lambda_{\alpha}\times \{0\}$ be vertices at the bottom
    level of the combinatorial horoball based on $\Lambda_{\alpha}$ in
    $\Gamma$ and let $(x,n), (y,n) \in \Lambda_{\alpha}\times \{n\}$ be the
    corresponding vertices at the $n$th level.  We have
    \[ \Bigl| d_{\mathcal{H}_n}\bigl((x,0),(y,0)\bigr) -
      d_{\mathcal{H}_n}\bigl((x,n),(y,n)\bigr) \Bigr| \le 2n \] by the
    triangle inequality.  It follows that the map
    $(\Lambda_{\alpha}^{(0)} \times {0}, d_{\mathcal{H}_n}) \to
    (\Lambda_{\alpha}^{(0)} \times \{n\}, d_{\mathcal{H}_n})$ given by
    $(x,0) \mapsto (x,n)$ is a rough isometry and
    $\Lambda_{\alpha}^{(0)} \times \{n\}$ is a convex subgraph of
    $\mathcal{H}_n$.
\end{proof}

We now recall and prove \cref{convexify_horoball}.
\convexify*

\begin{proof}
 Let $S$ be a finite generating set of $G$ containing each of the
  $S_i$.  Let $\Gamma$ be the Cayley graph of $G$ with respect to $S$.
  Then the Cayley graphs $\Gamma'_i = \Cay(H_i,S_i)$ are subgraphs of
  $\Gamma$ and $G$ is hyperbolic relative to the family
  $(g\Gamma'_i)_{g,i}$ of $G$-translates of these subgraphs.  By
  \cref{thm:convexify}, there is an $n$ for which the parabolics of
  $\mathcal{H}_n\bigl(\Gamma, (g\Gamma'_i)_{g,i}\bigr)$ are convex.
  For each $i$, let $\Gamma_i$ be the parabolic in the restricted
  horoball with base $\Gamma'_i$.  Then $\mathcal{H}_n$ and the
  $\Gamma_i$ satisfy all the required conditions.
\end{proof}

\section{A Cayley graph with strongly shortcut parabolics} \label{sec:cayley_with_ss_parabolics}

Let $G$ be a finitely generated group that is strongly shortcut relative to strongly shortcut subgroups $(H_i)_i$.  In this section we will show that there exists a generating set $S$ for $G$ such that the $H_i$ are strongly shortcut metric subspaces of the Cayley graph $\Cay(G,S)$.  In order to do this, we will first need to define what it means for a metric space to be strongly shortcut.  The following definition appears in 
earlier work of the first named author under the name \emph{nonapproximability of $n$-gons} \cite[Definition~3.2]{Hoda:shortcut_space}.

\begin{defn}
  Let $C_n$ denote the cycle graph of length $n$ (i.e., a circle
  subdivided into $n$ edges and $n$ vertices) and let $C_n^{(0)}$
  denote the vertex set of $C_n$.  A metric space $X$ is
  \defterm{strongly shortcut} if there exists a $K > 1$, an $n \in \N$
  and an $M > 0$ such that there is no $K$-bilipschitz embedding of
  $(C_n^{(0)}, \lambda d_{C_n})$ in $X$ with $\lambda \ge M$.
\end{defn}

\begin{thrm}[{\cite[Corollary~3.6]{Hoda:shortcut_space}}]
  \label{graph_metric_ss} A graph $\Gamma$ is strongly shortcut as a graph if and only if it is strongly shortcut as a metric space.
\end{thrm}

Our goal in this section is to prove the following.

\begin{thrm}
  \label{ss_parabolics_in_cayley} Let $G$ be a finitely generated group that is
  hyperbolic relative to a family of strongly shortcut groups
  $(H_i)_i$.  Then $G$ has a finite generating set $S$ for which the
  $H_i$ are strongly shortcut metric subspaces of $\Cay(G,S)$.
\end{thrm}

In order to prove \cref{ss_parabolics_in_cayley} we will rely on
\cref{thm:convexify} and the following refined version of the
Milnor-\v{S}varc Lemma.  This version of the Milnor-\v{S}varc Lemma
gives us arbitrary control on the multiplicative constant of the
quasi-isometry, up to scaling the metric on the Cayley graph.  This
arbitrary control on the multiplicative constant of the quasi-isometry
comes at the cost of having to choose larger and larger finite
generating sets and accepting larger and larger additive
quasi-isometry constants.

\begin{thrm}[{Fine Milnor-\v{S}varc Lemma \cite[Theorem~H]{Hoda:shortcut_space}}]
  \label{fine_ms}
  Let $(X,d)$ be a geodesic metric space.  Let $G$ be a group
  acting metrically properly and coboundedly on $X$ by isometries.
  Fix $x_0 \in X$.  For $t > 0$ let $S_t$ be the finite set defined by
  \[ S_t = \bigl\{ g \in G \sth d(x_0,gx_0) \le t \bigr\} \] and
  consider the word metric $d_{S_t}$ defined by $S_t$.  (For those $t$
  where $S_t$ does not generate $G$, we allow $d_{S_t}$ to take the
  value $\infty$).  Let $K_t$ be the infimum of all $K > 1$ for which
  \begin{align*}
    (G,td_{S_t}) &\to X \\
    g &\mapsto g \cdot x_0
  \end{align*}
  is a $(K, C_K)$-quasi-isometry for some $C_K \ge 0 $.  Then
  $K_t \to 1$ as $t \to \infty$.
\end{thrm}

\begin{lem}
  \label{approximating_parabolics}
  Let $G$ be a finitely generated group that is hyperbolic relative to
  finitely generated subgroups $(H_i)_i$.  For each $i$, let $S_i$ be
  a finite generating set for $H_i$.  Then, for any $L > 1$, there is
  a $t > 0$ and a finite generating set $S$ for $G$ such that each
  inclusion \[ (H_i, d_{S_i}) \hookrightarrow (G, td_S) \] is a
  quasi-isometric embedding with multiplicative constant $L$, where
  $d_S$ and the $d_{S_i}$ are the word metrics.
\end{lem}
\begin{proof}
  Let $S' \supseteq \bigcup_i S_i$ be a finite generating set for $G$.
  Let $\Gamma' = \Cay(G,S')$ and let $\Lambda_{g,i} = g\Cay(H_i,S_i)$.
  By \cref{thm:convexify}, for some $n$, the top level subgraphs $\Lambda_{g,i}\times \{n\}$ of the restricted horoballs of
  $\mathcal{H}_n = \mathcal{H}_n\bigl(\Gamma',\{\Lambda_{g,i}\}_{g,i}\bigr)$ are convex.  Moreover, by \cref{action_on_n_restricted_horoball}, the group $G$ acts properly and cocompactly on $\mathcal{H}_n$.

  By \cref{cor:rough_isometry_of_horopheres} and \cref{levels_rips}, there is a  rough isometry $(\Lambda_{e,i}^{(0)}\times \{0\}, d_{\mathcal{H}_n}) \to \bigl(H_i,\frac{1}{2^n}d_{S_i}\bigr)$.  By \cref{fine_ms}, there is a
  generating set $S$ for $G$ and a scaling factor $t' > 0$ such that
  the inclusion $(G, t' d_S) \hookrightarrow \mathcal{H}_n$ is a
  quasi-isometry with multiplicative constant $L$, where $d_S$ is the
  word metric coming from $S$.  But the image of $H_i$ under this inclusion is $\Lambda_{e,i}^{(0)}\times \{0\}$ and so the composition of the restriction $(H_i, t' d_S) \hookrightarrow (\Lambda_{e,i}^{(0)}\times \{0\}, d_{\mathcal{H}_n})$ and the rough isometry 
  $(\Lambda_{e,i}^{(0)}\times \{0\}, d_{\mathcal{H}_n}) \to \bigl(H_i,\frac{1}{2^n}d_{S_i}\bigr)$
  gives us a quasi-isometry $(H_i, t' d_S) \to \bigl(H_i,\frac{1}{2^n}d_{S_i}\bigr)$
  with multiplicative constant $L$.  Scaling the domain and the codomain by $2^n$, taking the quasi-inverse and composing it with the isometric embedding $(H_i, 2^n t' d_S) \hookrightarrow (G, 2^n t' d_S)$ we obtain a quasi-isometry $(H_i, d_{S_i}) \hookrightarrow (G, 2^n t' d_S)$ with multiplicative factor $L$.
\end{proof}

Finally, we will need the next two theorems about strongly shortcut spaces and groups.

\begin{thrm}[{\cite[Proposition~3.4]{Hoda:shortcut_space}}]
  \label{rough_approx_inv}
  Let $X$ be a strongly shortcut metric space.  Then
  there exists an $L_X > 1$ such that whenever $Y$ is a
  metric space and $C > 0$ and $f \colon Y \to X$ is an
  $(L_X,C)$-quasi-isometry up to scaling, then $Y$ is also strongly
  shortcut.
\end{thrm}

\begin{thrm}[{\cite[Theorem~C]{Hoda:shortcut_space}}]
  \label{ss_group}
  A group $G$ is strongly shortcut if and only if $G$ has a finite
  generating set $S$ for which $\Cay(G,S)$ is strongly shortcut.
\end{thrm}

\begin{proof}[Proof of \cref{ss_parabolics_in_cayley}]
  Let $G$ be a finitely generated group that is hyperbolic relative to
  strongly shortcut groups $(H_i)_i$.  By \cref{ss_group}, we can
  choose finite generating sets $S_i$ of $H_i$ so that the Cayley
  graphs $\Cay(H_i,S_i)$ are strongly shortcut.  Then, by
  \cref{rough_approx_inv}, for each $i$, there exists an $L_i > 1$
  such that any metric space that, up to scaling, is quasi-isometric
  to $(H_i, d_{S_i})$ with multiplicative constant $L_i$ is also
  strongly shortcut.  By \cref{approximating_parabolics}, there is a
  finite generating set $S$ of $G$ and a $t > 0$ such that, for each
  $i$, if $d_S$ is the word metric coming from $S$ then $(H_i, t d_S)$
  is quasi-isometric to $(H_i, d_{S_i})$ with multiplicative constant
  $L = \min_i L_i$.  Thus each $(H_i, d_S)$ is strongly shortcut.
\end{proof}

\section{Asymptotic cones and the proof of the main result} \label{sec:main_result}

In this section we will recall the definition of asymptotic cones of metric spaces.  Then we will state the theorem of Osin and Sapir on tree-gradedness of asymptotic cones of relatively hyperbolic groups and a theorem of the second named author giving an asymptotic cone characterization of the strong shortcut property.  We will use these theorems and the results of the previous sections to prove \cref{ss_closed_under_relhyp}.

For an exposition of asymptotic cones, see Drutu and Kapovich, 2018
\cite{Drutu:2018}.

\begin{defn}
A \defterm{non-principal ultrafilter} $\omega$ over $\mathbb{N}$ is a set of subsets of $\mathbb{N}$ satisfying the following properties:
\begin{enumerate}
\item For each $A \subseteq \mathbb{N}$, either $A \in \omega$ or $\mathbb{N} \setminus A \in \omega$, but not both.
\item No finite subset of $\mathbb{N}$ is in $\omega$.
\item If $A,B \in \omega$, then $A \cap B \in \omega$.
\item If $A \in \omega$ and $A \subseteq B$, then $B \in \omega$. 
\end{enumerate}
\end{defn}

The existence of non-principal ultrafilters is a consequence of Zorn's Lemma (see \cite[Lemma~10.18]{Drutu:2018} for instance).

\begin{defn}
Let $\omega$ be a non-principal ultrafilter. Let $(x_n)_n$ be a sequence of points in a topological space $X$. An element $x \in X$ is an \defterm{$\omega$-limit} of $(x_n)_n$, denoted $\lim_{\omega} x_n$, if for every open set $U \ni x$, the set $A_U = \{n \in \mathbb{N} | x_n \in U\}$ is contained in $\omega$. 
\end{defn}

\begin{rmk}
If $X$ is a Hausdorff space, then an $\omega$-limit is unique whenever it exists. If $X$ is compact, then for every sequence, an $\omega$-limit exists.
\end{rmk}

Let $(X,d)$ be a metric space and let $\omega$ be a non-principal ultrafilter over $\mathbb{N}$. Let $(r_n)_n$ be a sequence of real numbers such that $\lim_{\omega} r_n = \infty$. Fix a sequence of basepoints $(p_n)_n \in X^{\mathbb{N}}$.

Let $d_{\infty} : X^{\mathbb{N}} \times X^{\mathbb{N}} \to [0,\infty]$ be defined as $d_{\infty} ((x_n)_n, (y_n)_n) := \lim_{\omega} (d(x_n,y_n)/r_n)$. Let $X_B^{\mathbb{N}}((r_n)_n,(p_n)_n) := \{(x_n)_n \in X^{\mathbb{N}} | d_{\infty}((x_n)_n,(p_n)_n) < \infty\}$.

\begin{rmk}
Note that $(X_B^{\mathbb{N}}((r_n)_n,(p_n)_n),d_{\infty})$ is a pseudo-metric space.
\end{rmk}

\begin{defn}
  The \defterm{asymptotic cone} $\Cone_{\omega}(X,(r_n)_n,(p_n)_n)$ of $X$
  is the quotient of $X_B^{\mathbb{N}}((r_n)_n,(p_n)_n)$ identifying
  $(x_n)_n$ and $(y_n)_n$ whenever whenever $d_{\infty}((x_n)_n,(y_n)_n) = 0$.
  We let $[x_n]$ denote the point of $\Cone_{\omega}(X,(r_n)_n,(p_n)_n)$
  represented by $(x_n)_n$.
\end{defn}

\begin{rmk}
  For a group $G$ equipped with a left invariant metric, any
  asymptotic cone $\Cone_{\omega}(G,(r_n)_n,(p_n)_n)$ is isometric to
  $\Cone_{\omega}(G,(r_n)_n,(1)_n)$, where $(1)_n$ is the constant
  basepoint sequence at the identity.  Thus in this case we will
  simply write $\Cone_{\omega}(G,(r_n)_n)$.
\end{rmk}

\begin{defn}[Drutu and Sapir \cite{Drutu:2005}]\label{defn_tree_graded}
A complete geodesic metric space $X$ is a \defterm{tree graded space} with respect to a collection of closed geodesic subspaces, called \defterm{pieces}, if the following two properties are satisfied:
\begin{enumerate}
\item Any two distinct pieces intersect in at most a single point, and
\item Every non-trivial simple geodesic triangle (i.e., the concatenation of the three geodesics is a simple loop) in $X$ is contained in a piece.
\end{enumerate}
\end{defn}

\begin{thrm}[{Osin and Sapir \cite[Theorem~A.1]{Drutu:2005}}]
  \label{thm:treegraded} Let $G$ be a finitely generated group and let $d_S$ be a word
  metric coming from a finite generating set $S$ of $G$.  If $G$ is
  hyperbolic relative to a family of subgroups $(H_i)_i$ then every
  asymptotic cone
  $\mathcal{A} = \Cone_{\omega}\bigl((G,d_S),(r_n)\bigr)$ of
  $(G, d_S)$ is tree graded with respect to the $\omega$-limits
  \[ \lim_{\omega} (g_n H_i)_n = \bigl\{ [x_n]_n \in \mathcal{A} \sth
    x_n \in g_n H_i \bigr\} \] of the $(g_n H_i)_n$ with
  $[g_n]_n \in \mathcal{A}$.
\end{thrm}

\begin{rmk}
  \label{pieces_ascones}
  The $\lim_{\omega} (g_n H_i)_n$ are isometric to asymptotic cones of
    the $(H_i, d_S)$.  Indeed, the asymptotic cone $\mathcal{A}$ is a
    group with multiplication given by
    \[ [x_n]_n \cdot [y_n]_n = [x_ny_n]_n \] and $d_{\infty}$ is a
    left-invariant metric with respect to this group structure.  Thus
    $\lim_{\omega} (g_n H_i)_n$ is isometric to
    \[ [g_n^{-1}]_n\lim_{\omega} (g_n H_i)_n = \lim_{\omega} (H_i)_n \]
    which is $\Cone_{\omega}\bigl((H_i,d_S),(r_n)\bigr)$.
\end{rmk}

A \emph{Riemannian circle} $C$ is $S^1$ equipped with a geodesic metric
of some length $|C|$.  In other words $C$ is the quotient of $\R$ by
the action of $|C|\Z$.

\begin{thrm}[{\cite[Theorem~3.7]{Hoda:shortcut_space}}]
  \label{sshortcut_equiv}
  A metric space $X$ is strongly shortcut if and only
  if no asymptotic cone of $X$ contains an isometric copy of the
  Riemannian circle of unit length.
\end{thrm}

We are now ready to prove our main result, which we first recall:
\main*
\begin{proof}
  Let $G$ be a finitely generated group that is hyperbolic relative to strongly shortcut groups
  $(H_i)_i$.  By \cref{ss_parabolics_in_cayley}, there is a finite generating set $S$ of $G$ such that $(H_i, d_S)$ is strongly shortcut for each $i$, where $d_S$ is the word metric coming from $S$.  We will show that the Cayley graph $\Cay(G,S)$ is strongly shortcut.  By \cref{graph_metric_ss} and \cref{sshortcut_equiv} it will suffice to prove that no asymptotic cone $\mathcal{A}$ of $\Cay(G,S)$ contains a Riemannian circle of unit length.
  
  By \cref{thm:treegraded}, any embedded copy of
  $C$ in $\mathcal{A}$ is contained in some
  $\lim_{\omega} (g_n H_i)_n$ with $[g_n]_n \in \mathcal{A}$.  Thus it
  suffices to show that $\lim_{\omega} (g_n H_i)_n$ does not contain
  an isometric copy of the Riemannian circle of unit length.
  But by \cref{pieces_ascones}, the $\omega$-limit
  $\lim_{\omega} (g_n H_i)_n$ is isometric to an asymptotic cone
  $\mathcal{A}'$ of $(H_i, d_{S'})$, which is strongly shortcut.
  Hence $\mathcal{A}'$ cannot contain an isometric copy of the
  Riemannian circle of unit length, by \cref{sshortcut_equiv}.
\end{proof}

\bibliographystyle{alpha}
\bibliography{nima,relhyp}

\newcommand{\etalchar}[1]{$^{#1}$}
\begin{thebibliography}{CCG{\etalchar{+}}20}

\bibitem[BH99]{Bridson:1999}
Martin~R. Bridson and Andr\'{e} Haefliger.
\newblock {\em Metric spaces of non-positive curvature}, volume 319 of {\em
  Grundlehren der Mathematischen Wissenschaften [Fundamental Principles of
  Mathematical Sciences]}.
\newblock Springer-Verlag, Berlin, 1999.

\bibitem[BHS17]{hhs1}
Jason Behrstock, Mark~F. Hagen, and Alessandro Sisto.
\newblock Hierarchically hyperbolic spaces, {I}: {C}urve complexes for cubical
  groups.
\newblock {\em Geom. Topol.}, 21(3):1731--1804, 2017.

\bibitem[BHS19]{hhs2}
Jason Behrstock, Mark Hagen, and Alessandro Sisto.
\newblock Hierarchically hyperbolic spaces {II}: {C}ombination theorems and the
  distance formula.
\newblock {\em Pacific J. Math.}, 299(2):257--338, 2019.

\bibitem[Bow12]{bowditch_rh}
B.~H. Bowditch.
\newblock Relatively hyperbolic groups.
\newblock {\em Internat. J. Algebra Comput.}, 22(3):1250016, 66, 2012.

\bibitem[CCG{\etalchar{+}}20]{chalopin2020helly}
J{\' e}r{\' e}mie Chalopin, Victor Chepoi, Anthony Genevois, Hiroshi Hirai, and
  Damian Osajda.
\newblock {H}elly groups, 2020.
\newblock Preprint, arXiv:2002.06895.

\bibitem[Dah03]{dahmani_combination}
Fran\c{c}ois Dahmani.
\newblock Combination of convergence groups.
\newblock {\em Geom. Topol.}, 7:933--963, 2003.

\bibitem[DK18]{Drutu:2018}
Cornelia Dru\c{t}u and Michael Kapovich.
\newblock {\em Geometric group theory}, volume~63 of {\em American Mathematical
  Society Colloquium Publications}.
\newblock American Mathematical Society, Providence, RI, 2018.
\newblock With an appendix by Bogdan Nica.

\bibitem[DS05]{Drutu:2005}
Cornelia Dru\c{t}u and Mark Sapir.
\newblock Tree-graded spaces and asymptotic cones of groups.
\newblock {\em Topology}, 44(5):959--1058, 2005.
\newblock With an appendix by Denis Osin and Mark Sapir.

\bibitem[Far98]{farb_rh}
B.~Farb.
\newblock Relatively hyperbolic groups.
\newblock {\em Geom. Funct. Anal.}, 8(5):810--840, 1998.

\bibitem[GM08]{groves_manning_horoball}
Daniel Groves and Jason~Fox Manning.
\newblock Dehn filling in relatively hyperbolic groups.
\newblock {\em Israel J. Math.}, 168:317--429, 2008.

\bibitem[Gro87]{Gromov:1987}
M.~Gromov.
\newblock Hyperbolic groups.
\newblock In {\em Essays in group theory}, volume~8 of {\em Math. Sci. Res.
  Inst. Publ.}, pages 75--263. Springer, New York, 1987.

\bibitem[HHP20]{hhs_strongly_shortcut}
Thomas Haettel, Nima Hoda, and Harry Petyt.
\newblock The coarse {H}elly property, hierarchical hyperbolicity, and
  semihyperbolicity.
\newblock Preprint, arXiv:2009.14053, 2020.

\bibitem[HO19]{huang2019helly}
Jingyin Huang and Damian Osajda.
\newblock {H}elly meets {G}arside and {A}rtin, 2019.
\newblock Preprint, arXiv:1904.09060.

\bibitem[Hod17]{Hoda:2017}
Nima Hoda.
\newblock Quadric complexes.
\newblock {\em Preprint, arXiv:1711.05844}, 2017.

\bibitem[Hod18]{Hoda:2018b}
Nima Hoda.
\newblock Shortcut graphs and groups.
\newblock Preprint, arXiv:1811.05036, 2018.

\bibitem[Hod20]{Hoda:shortcut_space}
Nima Hoda.
\newblock Strongly shortcut spaces.
\newblock Preprint, 2020.

\bibitem[HP]{Hoda:heisenberg}
Nima Hoda and Piotr Przytycki.
\newblock The {H}eisenberg group is strongly shortcut.
\newblock In progress.

\bibitem[Hru10]{hruska_rh}
G.~Christopher Hruska.
\newblock Relative hyperbolicity and relative quasiconvexity for countable
  groups.
\newblock {\em Algebr. Geom. Topol.}, 10(3):1807--1856, 2010.

\bibitem[Kar11]{Kar:2011}
Aditi Kar.
\newblock Asymptotically {$\rm CAT(0)$} groups.
\newblock {\em Publ. Mat.}, 55(1):67--91, 2011.

\bibitem[MM99]{gcs1}
Howard~A. Masur and Yair~N. Minsky.
\newblock Geometry of the complex of curves. {I}. {H}yperbolicity.
\newblock {\em Invent. Math.}, 138(1):103--149, 1999.

\bibitem[MM00]{gcs2}
H.~A. Masur and Y.~N. Minsky.
\newblock Geometry of the complex of curves. {II}. {H}ierarchical structure.
\newblock {\em Geom. Funct. Anal.}, 10(4):902--974, 2000.

\bibitem[{Sis}12]{sisto_metric_rh}
Alessandro {Sisto}.
\newblock {On metric relative hyperbolicity}.
\newblock {\em arXiv e-prints}, page arXiv:1210.8081, October 2012.

\bibitem[Wis03]{Wise:2003}
Daniel~T. Wise.
\newblock Sixtolic complexes and their fundamental groups.
\newblock Preprint, 2003.

\end{thebibliography}

\end{document}